\documentclass[12pt]{amsart}

\usepackage{amssymb,latexsym}

\usepackage{enumerate}

 \usepackage[french,english]{babel}

\makeatletter

\@namedef{subjclassname@2010}{

  \textup{2010} Mathematics Subject Classification}

\makeatother
\newtheorem{thm}{Theorem}[section]
\newtheorem{Con}[thm]{Conjecture}

\newtheorem{cor}[thm]{Corollary}
\newtheorem{lem}[thm]{Lemma}
\newtheorem{pro}[thm]{Proposition}
\theoremstyle{definition}

\newtheorem{rem}[thm]{Remark}

\numberwithin{equation}{section}

\newcommand{\X}{\mathbb{X}}
\newcommand{\Y}{\mathbb{Y}}
\newcommand{\ex}{\mathbb{E}}
\newcommand{\re}{\textup{Re}}
\newcommand{\im}{\textup{Im}}
\newcommand{\pr}{\mathbb{P}}
\newcommand{\B}{\textup{Bi}_p}

\newcommand{\G}{\mathcal{G}}
\newcommand{\F}{\mathbb{F}_p}
\newcommand{\Fs}{\mathbb{F}_p^{\times}}
\newcommand{\M}{\mathcal{M}_p}

\newcommand{\newabstract}[1]{%
  \par\bigskip
  \csname otherlanguage*\endcsname{#1}%
  \csname captions#1\endcsname
  \item[\hskip\labelsep\scshape\abstractname.]
}

\begin{document}

\baselineskip=17pt

\title{On the distribution of the maximum of cubic exponential sums}

\author{Youness Lamzouri}

\address{Department of Mathematics and Statistics,
York University,
4700 Keele Street,
Toronto, ON,
M3J1P3
Canada}

\email{lamzouri@mathstat.yorku.ca}

\date{}

\begin{abstract}
In this paper, we investigate the distribution of the maximum of partial sums of certain cubic exponential sums, commonly known as ``Birch sums''. Our main theorem gives upper and lower bounds (of nearly the same order of magnitude) for the distribution of large values of this maximum, that hold in a wide uniform range. This improves a recent result of Kowalski and Sawin.  The proofs use a blend of probabilistic methods, harmonic analysis techniques, and deep tools from algebraic geometry. The results can also be generalized to other types of $\ell$-adic trace functions. In particular, the lower bound of our result also holds for partial sums of Kloosterman sums. As an application, we show that there exist $x\in [1, p]$ and $a\in \Fs$ such that $|\sum_{n\leq x} \exp(2\pi i (n^3+an)/p)|\ge (2/\pi+o(1)) \sqrt{p}\log\log p$. 
The uniformity of our results suggests that this bound is optimal, up to the value of the constant.


\end{abstract}


\thanks{The author is partially supported by a Discovery Grant from the Natural Sciences and Engineering Research Council of Canada.}

\maketitle

\section{Introduction}
Let $p\geq 3$ be a prime number, $p\nmid a$ be an integer, and $E$ be the elliptic curve over $\mathbb{Q}$ given by the Weierstrass equation $y^2=x^3+ax$. If we put $a_p=p+1-|E(\F)|$, then we have
\begin{equation}\label{Char}
a_p=\sum_{n\leq p} \chi_p(n^3+an),
\end{equation}
where $\chi_p$ is the Legendre symbol modulo $p$. Furthermore, we have the   Hasse bound $|a_p|\leq 2\sqrt{p}$. In \cite{Bi}, Birch proved the ``vertical'' Sato-Tate law for the $a_p$'s, which states that as $a$ varies in $\Fs$, the quantity $a_p/\sqrt{p}$ becomes equidistributed in $[-2, 2]$ with respect to the Sato-Tate measure 
$$ \mu_{ST}=\frac{1}{\pi} \sqrt{1-\frac{x^2}{4}}dx, $$
as $p\to \infty$. In analogy with the multiplicative character sum \eqref{Char}, Birch \cite{Bi} also conjectured that a similar result should hold for the normalized cubic exponential sum
\begin{equation}\label{BirchSum}
\B(a):=\frac{1}{\sqrt{p}}\sum_{n\in \F} e_p\left(n^3+an\right),
\end{equation}
where here and throughout we let $e(z):=\exp(2\pi i z)$, and $e_p(z):=e(z/p)$ is the standard additive character modulo $p$. The sums $\B(a)$ are commonly known as \emph{Birch sums}.  Birch's conjecture was subsequently proved by Livn\'e in \cite{Li}. 

Recently, Kowalski and Sawin \cite{KoSa} investigated the distribution of the polygonal paths formed by linearly interpolating the partial sums 
\begin{equation}\label{PS}
\frac{1}{\sqrt{p}}\sum_{0\leq n\leq x} e_p\left(n^3+an\right),
\end{equation}
for $ 0\leq x\leq p-1$. Let $$\M(a)=\mathcal{M}_{\B}(a):=\max_{x<p} \frac{1}{\sqrt{p}} \left|\sum_{0\leq n\leq x} e_p(n^3+an)\right|.$$
Among their results, Kowalski and Sawin proved that as $a$ varies in $\Fs$ and $p\to \infty$, the quantity $\M(a)$ converges in law to the random variable 
\begin{equation}\label{ProbModel}
 \mathbb{M}=\max_{\alpha\in [0, 1)} \left| \alpha \X(0)+ \sum_{h\neq 0} \frac{e(\alpha h)-1}{2\pi i h} \X(h)\right|,
 \end{equation}
where $\{\X(h)\}_{h\in \mathbb{Z}}$ is a sequence of independent random variables with Sato-Tate distributions on $[-2, 2]$.  The proof uses deep results of Deligne, Katz, Laumon and others concerning the ramification and monodromy groups of certain sheaves associated to Birch sums. The origin of the probabilistic model \eqref{ProbModel} comes from the following identity, which is an immediate consequence of the discrete Plancherel formula 
\begin{equation}\label{Planche}
\frac{1}{\sqrt{p}} \sum_{0\leq n\leq x} e_p(n^3+an)=\frac{1}{\sqrt{p}}\sum_{|h|<p/2}\gamma_p(h;x) \B(a-h),
\end{equation}
where $$\gamma_p(h;x):= \frac{1}{\sqrt{p}} \sum_{0\leq m\leq x}e_p\left(mh\right) $$
are the Fourier coefficients modulo $p$ of the characteristic function of the interval $[0, x]$. Furthermore, one has the elementary estimate (see page 13 of \cite{KoSa})
\begin{equation}\label{Planche2}
\frac{1}{\sqrt{p}}\gamma_p(h;x)= \frac{e_p\left(xh\right)- 1}{2\pi i h} +O\left(\frac{1}{p}\right),
\end{equation}
which holds uniformly for $1\leq |h|<p/2$. Therefore, assuming that the Birch sums $\B(a-h)$ behave ``independently'' for different shifts $h$ and are all Sato-Tate distributed on $[-2, 2]$, shows that $\mathbb{M}$ is a good model for $\M(a)$. 

For a positive real number $V$, we define
$$ \Phi_p(V):= \frac{1}{p-1} \left|\{ a\in \Fs: \M(a)>V\}\right|.$$
Kowalski and Sawin \cite{KoSa} proved that the limiting distribution of $\Phi_p(V)$ is double exponentially decreasing. More precisely, they showed that there exists a constant $c>0$ such that  
\begin{equation}\label{LargeDev}
c^{-1}\exp\left(-\exp\left(cV)\right)\right)\leq \lim_{p\to \infty} \Phi_p(V) \leq c \exp\left(-\exp\left(c^{-1}V)\right)\right).
\end{equation}
In this paper, we shall study the distribution of large values of $\M(a)$ as $a$ varies in $\Fs$. Our main result is the following theorem which substantially improves the estimate \eqref{LargeDev}.

\begin{thm}\label{Main}
Let $p$ be a large prime. For all real numbers $1\ll V\leq (2/\pi)\log\log p-2\log\log\log p$ we have 
$$ \exp\left(-A_0 \exp\left(\frac{\pi}{2} V\right) \left(1+O\left(\sqrt{V} e^{-\pi V/4}\right)\right)\right) \leq \Phi_p(V) \leq \exp\left(-C \exp\left(\left(\frac{\pi}{2}-\delta\right)V\right)\right)
$$
for some positive constant $C$, where 
$$ \delta= \frac{4\pi-\pi^2}{2\pi +8}=0.18880..., \text{ and } A_0= \exp\left(-\gamma-1 -\frac{1}{2} \int_0^{\infty} \frac{f(u)}{u^2}du \right),$$
where $\gamma$ is the Euler-Mascheroni constant, and $f:[0, \infty)\to \mathbb{R}$ is defined by
$$f(t):=\begin{cases} \log \ex (e^{t\X}) & \text{ if } 0\leq t <1,\\ \log \ex (e^{t\X})- 2t & \text{ if } t\geq 1.\end{cases}$$
where $\X$ is a random variable with Sato-Tate distribution  $\mu_{ST}$.
\end{thm}

\begin{rem}
The upper bound of Theorem \ref{Main} is valid in the extended range $1\ll V \leq (\log\log p)/(\pi/2-\delta)-2\log\log\log p$. It would be interesting to obtain a more precise estimate for $\Phi_p(V)$. The analogy with character sums (see the discussion on page 4) lead us to believe that the true order of magnitude of $\Phi_p(V)$ is perhaps closer to the lower bound of Theorem \ref{Main}.
It is also curious to note that the constant $\frac{1}{2} \int_0^{\infty} \frac{f(u)}{u^2}du$ appears in an asymptotic estimate of Liu-Royer-Wu \cite{LRW}, for the distribution function of large (or small) values of $L$-functions attached to holomorphic cusp forms at $s=1$.
\end{rem}

Kowalski and Sawin also investigated the polygonal paths formed by the linear interpolation of partial sums of  Kloosterman sums and obtained a similar result to \eqref{LargeDev} in this case. The lower bound of Theorem \ref{Main} holds verbatim for the maximum of partial sums of Kloosterman sums, but the proof of the upper bound fails in this case. Indeed, one of the main ingredients in this proof are strong bounds for short sums of cubic exponential sums, which are not currently known for Kloosterman sums. More precisely, in order to carry out the proof of the upper bound of Theorem \ref{Main} to this setting one needs the following bound
\begin{equation}\label{Kloosterman}
 \sum_{n\in I} e_p (ax+ \overline{x})\ll |I|^{1-\varepsilon},
\end{equation}
 for any interval $I\subset [1, p]$ of length $p^{1/2+\varepsilon/2}$ for some $\varepsilon>0$, where $\overline{x}$ is the multiplicative inverse of $x$ modulo $p$. 

Our result should be compared with the recent work of Bober, Goldmakher, Granville and Koukoulopoulos \cite{BGGK} concerning the distribution of the maximum of character sums. The proof of the upper bound of Theorem \ref{Main} follows the same strategy as \cite{BGGK} but uses several different ingredients, while the proof of the lower bound is completely different. This is mainly due to the lack of multiplicativity in our case which plays a central role for character sums. In particular, the analogue of the lower bound of Theorem \ref{Main} in \cite{BGGK} follows readily by relating character sums to values of Dirichlet $L$-functions and using the work of Granville-Soundararajan \cite{GS} on the distribution of $L(1, \chi)$. Another crucial difference in our case is the fact that the Birch sums $\B(a-h)$ and $\B(a+h)$ behave independently, while this is clearly not the case for the values $\chi(-h)$ and $\chi(h)$ where $\chi$ is a Dirichlet character. This makes the analysis of the exponential sum $\sum_{|h|\leq H} \frac{e(\alpha h)-1}{h} \B(a-h)$ more complicated in our case, which explains why our Theorem \ref{Main} is less precise than Theorem 1.1 of \cite{BGGK}. 
However, our probabilistic model is easier to work with, due to the fact that the $\X(h)$ are independent (in the case of character sums, the $\X(h)$ are multiplicative random variables that are independent for different primes). This is exploited in the proof of the lower bound of Theorem \ref{Main} through relating the Laplace transform of the sum $\im \frac{1}{\sqrt{p}} \sum_{n\leq p/2} e_p(n^3+an)$ to that of its corresponding random model, and using the saddle-point method to obtain precise estimates for the distribution of its large values (see Section 6 below).

The Birch sums \eqref{BirchSum} are examples of $\ell$-adic trace functions over finite fields. These trace functions have been extensively studied in a series of recent works by Fouvry, Kowalski, and Michel \cite{FKM1}, \cite{FKM2}, \cite{FKM3}, Fouvry,  Kowalski, Michel, Raju, Rivat, and Soundararajan \cite{FKMRRS}, Kowalski and Sawin \cite{KoSa}, \cite{KoSa2}, and Perret-Gentil \cite{Pe1}, \cite{Pe2}. Our results can be generalized to other types of trace functions that are attached to certain \emph{coherent} families of $\ell$-adic sheaves (in the sense given by Perret-Gentil \cite{Pe1}), if their short sums satisfy a bound similar to \eqref{Kloosterman}. The precise definition of a coherent family is technical (see \cite{Pe1}), but roughly speaking, these are sheaves over $\F$ for which the ``conductor'' is bounded independently of $p$, the arithmetic and geometric monodromy groups are equal, of fixed type and large, and the sheaves formed by additive shifts are ``independent''. As an example, Theorem \ref{Main} can be generalized for the  partial sums of the exponential sum 
\begin{equation}\label{GBirch}
\frac{1}{\sqrt{p}}\sum_{n\in \F} e_p\left(an+ f(n)\right),
\end{equation}
where $f\in \mathbb{Z}[t]$ is an odd polynomial of degree $n\geq 3$ with $n\neq 7, 9$. In this case the sums \eqref{GBirch} are real valued, and the monodromy group of the associate sheaf is $\textup{Sp}_{n-1}(\mathbb{C})$. 

As a corollary of the lower bound of Theorem \ref{Main} (more precisely of Theorem \ref{LowerB} below), we exhibit large values of partial sums of Birch sums. The same result also holds for partial sums of Kloosterman sums. 
\begin{cor}\label{Main2} Let $p$ be a large prime. There exist at least $p^{1-1/\log\log p}$ points $a\in \Fs$ such that 
$$ \left|\sum_{n\leq p/2} e_p(n^3+an)\right|\ge \left(\frac{2}{\pi}+o(1)\right) \sqrt{p}\log\log p.$$

\end{cor}
\begin{rem}
Using a completely different method, Bonolis \cite{Bo}  independently proved that $\max_{x<p}\left|\sum_{n\leq x} e_p(n^3+an)\right| \geq c\sqrt{p}\log\log p$ for at least $p^{1-\varepsilon}$ points $a\in \Fs$, though with a much smaller positive constant $c$. He also obtained the same result for partial sums of Kloosterman sums. 
\end{rem}

The classical Weil bound for exponential sums implies that $|\B(a)|\leq 2$. Using this bound in \eqref{Planche} we obtain the analogue of the P\'olya-Vinogradov inequality for $\M(a)$, namely that
$$ \max_{x<p}\left|\sum_{n\leq x} e_p(n^3+an)\right| \ll \sqrt{p} \log p, $$
uniformly for $a\in \Fs$. The double exponential decay of the distribution $\Phi_p(V)$ and the uniformity of Theorem \ref{Main} lead us to formulate the following stronger conjecture, which is optimal up to the value of the constant by Corollary \ref{Main2}.
\begin{Con}\label{Incomplete}
There exists a positive constant $C_0$, such that for all primes $p\geq 3$ and all $1<x\leq p$ we have
$$ \left|\sum_{n\leq x} e_p(n^3+an)\right| \leq C_0 \sqrt{p} \log\log p.$$
\end{Con}
Montgomery and Vaughan \cite{MV} proved the analogue of this conjecture for character sums assuming the generalized Riemann hypothesis (GRH) for Dirichlet $L$ functions. It would be interesting to prove Conjecture \ref{Incomplete} conditionally on some unproven but widely believed hypotheses such as the GRH. 

The paper is organized as follows: In Section 2 we give an outline and present the main ingredients of the proof of the upper bound of Theorem \ref{Main}. In Section 3 we  use harmonic analysis techniques to obtain a non-trivial bound for a ``random'' exponential sum. In Section 4, we investigate the moments of sums of Birch sums using ingredients from algebraic geometry.  The proof of the upper bound of Theorem \ref{Main} will then be completed in Section 5. Finally, in Section 6 we prove the lower bound of Theorem \ref{Main}. 

\

\textbf{Acknowledgments.} 
 I would like to thank Corentin Perret-Gentil for useful comments and suggestions concerning the generalization of Theorem \ref{Main} to other trace functions. I thank Dante Bonolis for informing me about his current work on the size and moments of incomplete Kloosterman and Birch sums. I also thank Alexey Kuznetsov for helpful discussions.

\section{Proof the upper bound in Theorem \ref{Main}: Strategy and key ingredients}

We shall first transform the problem of bounding $\M(a)$ into a discrete problem involving bounding the maximum of  the partial sums $\sum_{n\leq x} e_p(n^3+an)$ over  a ``small'' number of the $x$'s. This is accomplished using the following bound for short exponential sums which follows from Weyl's method (see for example Lemma 20.3 of \cite{IK})
\begin{equation}\label{ShortExponential}
\sum_{n\in I}e_p(n^3+an) \ll_{\varepsilon} |I|^{1/4+\varepsilon} p^{1/4}, 
\end{equation}
for any interval $I\subset [1, p]$. We split the interval $[1, p]$ into $L$ intervals $I_{\ell}:=[x_{\ell}, x_{\ell+1}]$ where $x_0:=1$, $x_{L}:=p$, and for $1\leq \ell \leq L-1$, we define 
\begin{equation}\label{xl}
x_{\ell}:=\frac{\ell p}{L}.
\end{equation}
Using \eqref{ShortExponential} we prove
\begin{lem}\label{ReductionShort}
Let $p$ be a large prime and $L=p^{1/8}$. For all $a\in \Fs$, we have 
$$ \M(a)= \max_{0\leq \ell \leq L-1} \frac{1}{\sqrt{p}} \left|\sum_{0\leq n\leq x_{\ell}} e_p(n^3+an)\right| + O\left(p^{-1/50}\right).$$ 
\end{lem}

\begin{proof} The implicit lower bound is trivial, so it remains to prove the implicit upper bound. For each $a\in \Fs$ let $j(a)$ be that integer in $[1, p]$ for which 
$$\M(a)=\frac{1}{\sqrt{p}} \left|\sum_{0\leq n\leq j(a)} e_p(n^3+an)\right|.$$
Then $j(a)\in I_{\ell}$ for some $0\leq \ell \leq L-1$, and hence we have 
\begin{equation}\label{BoundM}
\M(a)\leq \frac{1}{\sqrt{p}} \left|\sum_{0\leq n\leq x_{\ell}} e_p(n^3+an)\right|+ \frac{1}{\sqrt{p}} \left|\sum_{x_{\ell}\leq n\leq j(a)} e_p(n^3+an)\right|.
\end{equation}
Now, we use the bound \eqref{ShortExponential} to obtain
$$ \frac{1}{\sqrt{p}} \left|\sum_{x_{\ell}\leq n\leq j(a)} e_p(n^3+an)\right| \ll_{\varepsilon} p^{-1/4} |I_{\ell}|^{1/4+\varepsilon}\ll p^{-1/50}.$$
Inserting this estimate in \eqref{BoundM} completes the proof.
\end{proof}

Combining Lemma \ref{ReductionShort} with equations \eqref{Planche} and \eqref{Planche2} we obtain 
\begin{equation}\label{Upper2}
\M(a)
\leq  
\frac{1}{2\pi} \max_{0\leq \ell \leq L-1} \left|\sum_{1\leq |h|<p/2} \frac{e_p\left(x_{\ell} h\right)- 1}{h} \B(a-h)\right| +O(1).
\end{equation}
In order to bound the distribution function of $\M(a)$, a standard approach is to bound the moments of $\max_{0\leq \ell \leq L-1} \left|\sum_{1\leq |h|<p/2} \frac{e_p\left(x_{\ell} h\right)- 1}{h} \B(a-h)\right|$. However, it turns out that a more efficient method is to truncate this sum at a parameter $1\leq H\leq p$, and treat the terms $\B(a-h)$ for $1\leq |h|\leq H$ as random points in $[-2, 2]$. This gives \begin{equation}\label{Upper}
\M(a)\leq  
\frac{\mathcal{G}(H)}{2\pi} + \frac{1}{2\pi} \max_{0\leq \ell \leq L-1} \left|\sum_{H<|h|<p/2} \frac{e_p\left(x_{\ell} h\right)- 1}{ h} \B(a-h)\right| +O(1),
\end{equation}
where 
$$\mathcal{G}(H):=\max_{\alpha\in [0, 1)} \max_{(y_{-H}, \dots y_{-1}, y_1,\dots,  y_{H})\in [-2, 2]^{2H}}
 \left|\sum_{1\leq |h|\leq H} \frac{e(\alpha h)-1}{h} y_h\right|.
 $$
In Section 3, we will investigate the quantity $\G(H)$ and obtain a non-trivial upper bound for it. More precisely, we shall prove
\begin{thm}\label{BOUNDGH}
Let $H$ be a positive integer. Then, we have 
$$ \G(H)\leq \left(2+\frac{8}{\pi}\right) \log H+O(1).$$
\end{thm}
It remains now to bound the moments of the ``tail''
\begin{equation}\label{TailBirch}
\max_{0\leq \ell \leq L-1} \left|\sum_{H<|h|<p/2} \frac{e_p\left(x_{\ell} h\right)- 1}{ h} \B(a-h)\right|.
\end{equation}
To this end, we shall use the recent work of Perret-Gentil \cite{Pe1} which relies on deep tools from algebraic geometry, in order to investigate the moments of sums of the form $\sum_{y\le |h|<z} c(h) \B(a-h)$, where $\{c(h)\}_{h\in \mathbb{Z}}$ is a sequence of complex numbers. This will be carried out in Section 4. 
Using these results, we shall establish the following theorem in Section 5.
\begin{thm}\label{MomentsMaxTail}
Let $p$ be a large prime, and $k$ be a large positive integer such that $k\leq (\log p)/(100\log\log p)$. Let $S$ be a non-empty subset of $[0, 1)$ such that $|S|\leq \sqrt{p}$, and put $y=10^5 k$. Then we have
$$ 
\frac{1}{p-1}\sum_{a\in \Fs} \max_{\alpha\in S}\left|\sum_{y\le |h|<p/2} \frac{e\left(\alpha h\right)- 1}{h} \B(a-h)\right|^{2k} \ll e^{-2k} + \frac{|S|(4\log p)^{10k}}{\sqrt{p}}.
$$
\end{thm}
This result gives non-trivial bounds for these moments only when $|S|\leq p^{1/2-\varepsilon}$ for some $\varepsilon>0$. Thus, in view of Lemma \ref{ReductionShort} we need strong bounds for short exponential sums, which are supplied in our case by \eqref{ShortExponential}. Such bounds are not currently known for  Kloosterman sums. However, if we assume \eqref{Kloosterman} then the proof of the upper bound of Theorem \ref{Main} will extend to this case. 

With Theorems \ref{BOUNDGH} and \ref{MomentsMaxTail} now in place, we are ready to prove the upper bound of Theorem \ref{Main}.

\begin{proof}[Proof of the upper bound of Theorem \ref{Main}]
First, it follows from \eqref{Upper2} that there exists a set $S\subset [0, 1)$ with $|S|=p^{1/8}$ such that 
$$ \M(a)\leq \frac{1}{2\pi} \max_{\alpha\in S} \left|\sum_{1\leq |h|<p/2} \frac{e\left(\alpha h\right)- 1}{h} \B(a-h)\right| +O(1).$$
Let $k \leq (\log p)/(100\log\log p)$ be a large positive integer to be chosen, and put $y=10^5k$. Then, it follows from Theorem \ref{BOUNDGH} that 
$$ \M(a) \leq \left(\frac{1}{\pi}+\frac{4}{\pi^2}\right)\log k+ \frac{1}{2\pi} \max_{\alpha\in S} \left|\sum_{y\leq |h|<p/2} \frac{e\left(\alpha h\right)- 1}{h} \B(a-h)\right| +C_0,$$
for some positive constant $C_0$. Note that $1/(1/\pi +4/\pi^2))= \pi/2-\delta$.  We now choose $k=[C_1\exp((\pi/2-\delta)V)]$, where $C_1= \exp(-4 (C_0+1))$. Then, it follows from Theorem \ref{MomentsMaxTail} that\begin{align*}
\Phi_p(V) 
&\leq \frac{1}{p-1} \left|\left\{a\in \Fs: \max_{\alpha\in S} \left|\sum_{y\leq |h|<p/2} \frac{e\left(\alpha h\right)- 1}{h} \B(a-h)\right|\geq 1\right\}\right|\\
& \leq \frac{1}{p-1} \sum_{a\in \Fs} \max_{\alpha\in S} \left|\sum_{y\leq |h|<p/2} \frac{e\left(\alpha h\right)- 1}{h} \B(a-h)\right|^{2k} \\
& \ll e^{-2k} + \frac{(4\log p)^{10k}}{p^{3/8}} \ll \exp\left(-C_1\exp\left(\left(\frac{\pi}{2}-\delta\right)V\right)\right),
\end{align*}
which completes the proof.
\end{proof}

\section{A non-trivial upper bound for $\mathcal{G}(H)$: proof of Theorem \ref{BOUNDGH} }\label{GH}
Recall that
$$\mathcal{G}(H)=\max_{\alpha\in [0, 1)} \max_{(y_{-H}, \dots y_{-1}, y_1,\dots,  y_{H})\in [-2, 2]^{2H}}
 \left|\sum_{1\leq |h|\leq H} \frac{e(\alpha h)-1}{h} y_h\right|.
 $$
One can easily drive the following ``trivial'' bounds 
\begin{equation}\label{TrivialBound} 4\log H +O(1)\leq  \G(H) \leq 8 \log H +O(1),
\end{equation}
where the lower bound is obtained by taking $\alpha=1/2$, and the upper bound follows from the fact that 
$|(e(\alpha h) -1) y_h|\leq 4$. It is an interesting problem to obtain an asymptotic formula for $\G(H)$ as $H\to \infty$. The purpose of this section is to prove Theorem \ref{BOUNDGH},  which gives a non-trivial upper bound for $\G(H)$. We start with the following lemma.
 \begin{lem}\label{FHg}
 Let $H$ be a positive integer. Then, we have 
 $$ \G(H)\leq 4\max_{\alpha\in [0, 1)}\sum_{1\leq h\leq H} \frac{g(2\pi\alpha h)}{h}.$$
 where $g$ is the $2\pi$-periodic non-negative continuous function defined on $[0, 2\pi]$ by 
$$ 
g(t):=\begin{cases} \sin(t) & \text{ if } 0\leq t\leq \pi/2,\\ 
1-\cos(t) & \text{ if } \pi/2<t< 3\pi/2,\\ -\sin(t) & \text{ if } 3\pi/2\leq t\leq 2\pi.
\end{cases}
$$
\end{lem}
\begin{proof}
Let $\alpha \in [0, 1)$ and $(y_{-H}, \dots y_{-1}, y_1,\dots,  y_{H})\in [-2, 2]^{2H}$. Then, we have 
\begin{equation}\label{Grouph}
\begin{aligned}
 \left|\sum_{1\leq |h|\leq H} \frac{e(\alpha h)-1}{h} y_h\right|& = \left|\sum_{1\leq h\leq H}\left( \frac{e(\alpha h)-1}{h} y_h+ \frac{1-e(-\alpha h)}{h} y_{-h}\right)\right| \\
 & \leq \sum_{1\leq h\leq H} \frac{|f_{2\pi \alpha h}(y_h, y_{-h})|}{h},
 \end{aligned}
 \end{equation}
 where
$$ f_{\beta}(x, y) = (e^{i\beta}-1)x+(1-e^{-i\beta})y,$$
for $\beta\in \mathbb{R}$ and $(x, y)\in [-2, 2]^2$. Morover, we note that
\begin{align*}
|f_{\beta}(x, y)|^2&= (\cos(\beta)-1)^2(x-y)^2+\sin(\beta)^2 (x+y)^2\\
&= 2(x^2+y^2) (1-\cos(\beta))+ 4xy \cos(\beta) (1-\cos(\beta)).
\end{align*}
Therefore, if $(x, y) \in [-2, 2]$ and $\cos(\beta)\geq 0$ then
$$|f_{\beta}(x, y)|^2 \leq 16 (1-\cos(\beta))+ 16\cos(\beta) (1-\cos(\beta))= 16 \sin(\beta)^2,$$ while
if $\cos(\beta)<0$, then
$$|f_{\beta}(x, y)|^2 \leq 16 (1-\cos(\beta))- 16\cos(\beta) (1-\cos(\beta))= 16 (1-\cos(\beta))^2.$$
Thus, in both case we deduce that $|f_{\beta}(x, y)|\leq 4 g(\beta)$ for all $(x, y)\in [-2, 2]$. Inserting this bound in \eqref{Grouph} completes the proof.
\end{proof}
To estimate the sum on the right hand side of Lemma \ref{FHg} we shall use the Fourier series expansion of the function $g$. Let $a_n$, $b_n$ be the Fourier coefficients of $g$, defined by 
$$ a_n:=\frac{1}{\pi} \int_{-\pi}^{\pi} g(t) \cos(nt) dt \text{ for } n\geq 0,$$ and 
$$b_n:=\frac{1}{\pi} \int_{-\pi}^{\pi} g(t) \sin(nt) dt \text{ for } n\geq 1.
$$
Since $g$ is even we have $b_n=0$ for all $n\geq 1$, and 
\begin{align*}
a_n &= \frac{2}{\pi} \int_{0}^{\pi} g(t) \cos(nt) dt\\
&=\frac{2}{\pi}\int_{0}^{\pi/2}\sin(t) \cos(nt)dt+\frac{2}{\pi}\int_{\pi/2}^{\pi} \cos(nt)dt - \frac{2}{\pi}\int_{\pi/2}^{\pi}\cos(t) \cos(nt)dt. 
\end{align*}
When $n=0$ we have  
$$ a_0= \frac{2}{\pi}\int_{0}^{\pi/2}\sin(t)dt+ 1 - \frac{2}{\pi}\int_{\pi/2}^{\pi}\cos(t) dt= 1+\frac{4}{\pi},$$
while for $n\geq 1$ we have 
\begin{align*}
a_n&= \frac{1}{\pi}\int_{0}^{\pi/2}\big(\sin((n+1) t)-\sin((n-1)t)\big)dt-\frac{2\sin(n\pi/2)}{n \pi}\\ 
&  \ \ \ \ -\frac{1}{\pi}\int_{\pi/2}^{\pi}\big(\cos((n+1) t)+\cos((n-1)t)\big)dt.\\
\end{align*}
Hence, an easy calculation shows that $a_1= -\frac{1}{\pi}-\frac{1}{2}$ and 
and for $n\geq 2$ we have
\begin{equation}\label{FourierCoefficients}
\begin{aligned}
a_n
&= \frac{1-\cos((n+1)\pi/2)}{(n+1)\pi} - \frac{1-\cos((n-1)\pi/2)}{(n-1)\pi} -\frac{2\sin(n\pi/2)}{n \pi}\\ &  \ \ \ \ + \frac{\sin((n+1)\pi/2)}{(n+1)\pi} +\frac{\sin((n-1)\pi/2)}{(n-1)\pi}\\
&=\begin{cases} -\frac{4}{(n^2-1)\pi} & \text{ if } n\equiv 0 \bmod 4,\\
-\frac{2}{n(n+1) \pi} & \text{ if } n\equiv 1 \bmod 4,\\
0 & \text{ if } n\equiv 2 \bmod 4,\\
-\frac{2}{n(n-1) \pi} & \text{ if } n\equiv 3  \bmod 4.\\
\end{cases}
\end{aligned}
\end{equation}
Finally, since $a_n\ll 1/n^2$ for all $n\geq 1$ we have $\sum_{n\geq 1}|a_n|<\infty$, which implies that uniformly for $t\in \mathbb{R}$ we have
\begin{equation}\label{Fourier}g(t) = \frac{a_0}{2}+ \sum_{n=1}^{\infty} a_n\cos(nt).
\end{equation}
 
For $t\in \mathbb{R}$, let $||t||$ be the distance from $t$ to the nearest integer. Using the Fourier series expansion \eqref{Fourier}, we shall obtain an asymptotic estimate for the sum $\sum_{h\leq H} g(2\pi\alpha h)/h$, which depends on whether $\alpha$ is close to a rational number of small denominator.
\begin{lem}\label{MinorMajor}
Let $H$ be large, and $R=\log H$. Then, for any $\alpha\in [0, 1)$ such that $r\alpha\notin \mathbb{Z}$ for all $r\leq R$, we have 
\begin{equation}\label{MinorArcs2} 
\sum_{h\leq H} \frac{g(2\pi\alpha h)}{h} = \frac{a_0}{2} \log H -\sum_{1\leq r\leq R}a_r \log |1-e(r\alpha )| +O\left(1+\frac{1}{H} \sum_{1\leq r\leq R}\frac{|a_r|}{||r\alpha||} \right),
\end{equation}
where the $a_r$ are defined by \eqref{FourierCoefficients}.
Furthermore, if $\alpha=b/\ell$ where $(b, \ell)=1$ and $\ell\leq R$ then 
\begin{equation}\label{MajorArcs2}
\sum_{h\leq H} \frac{g(2\pi\alpha h)}{h} = \left(\frac{a_0}{2}  +\sum_{1\leq m\leq R/\ell} a_{m\ell}\right) \log H- \sum_{\substack{1\leq r\leq R\\ \ell\nmid r}}a_r \log |1-e(r\alpha )| +O(1).
\end{equation}
\end{lem}

\begin{proof}
Since $a_n\ll 1/n^2$ for all $n\geq 1$, we deduce from \eqref{Fourier} that 
$$ g(t)= \frac{a_0}{2}+ \sum_{1\leq r \leq R} a_r \cos(rt)+ O\left(\frac{1}{R}\right), $$
uniformly for $t\in \mathbb{R}$. This gives
\begin{equation}\label{LogAverage}
 \sum_{h\leq H} \frac{g(2\pi\alpha h)}{h}= \frac{a_0}{2}\log H+ \sum_{1\leq r\leq R}a_r \sum_{h\leq H} \frac{\cos(2\pi r\alpha  h)}{h}+ O\left(1 \right). 
\end{equation}
Now if $r\alpha\notin \mathbb{N}$, then for any positive integer $N$ we have 
$$
\sum_{h\leq N} e(r\alpha h)= \frac{e((N+1)r\alpha)-1}{e(r\alpha)-1}\ll \frac{1}{||r\alpha||}.
$$
Hence, combining this bound with partial summation we obtain
$$
\sum_{h> H} \frac{e(r\alpha h)}{h}\ll \frac{1}{||r\alpha||H}.
$$
Thus, if $r\alpha\notin \mathbb{N}$ we deduce that 
\begin{equation}\label{BoundExp4}\sum_{h\leq H} \frac{\cos(2\pi r\alpha  h)}{h} = \re \sum_{h=1}^{\infty} \frac{e(r\alpha h)}{h} + O\left(\frac{1}{||r\alpha||H}\right)= -\log|1-e(r\alpha)|+O\left(\frac{1}{||r\alpha||H}\right).
\end{equation}
Inserting this estimate in \eqref{LogAverage} completes the proof of \eqref{MinorArcs2}. 

Now, suppose that $\alpha=b/\ell$ where $(b, \ell)=1$ and $\ell\leq R$. If $\ell\mid r$ then 
$$\sum_{h\leq H} \frac{\cos(2\pi r\alpha  h)}{h}= \log H+O(1).$$
On the other hand if $\ell\nmid r$, then $||r\alpha||\geq 1/\ell\geq 1/R$. Hence, it follows from \eqref{BoundExp4} that in this case we have
$$\sum_{h\leq H} \frac{\cos(2\pi r\alpha  h)}{h}= -\log|1-e(r\alpha)|+O\left(\frac{\log H}{H}\right).$$
The proof of \eqref{MajorArcs2} follows upon combining these estimates with \eqref{LogAverage}. 
\end{proof}
We are now ready to prove Theorem \ref{BOUNDGH}.
\begin{proof}[Proof of Theorem \ref{BOUNDGH}]
By Lemma \ref{FHg} it suffices to prove that for all $\alpha\in [0, 1)$ we have
\begin{equation}\label{BoundSumgh}
\sum_{h\leq H} \frac{g(2\pi\alpha h)}{h} \leq \frac{a_0}{2}\log H+O(1). 
\end{equation}
Let $\alpha\in [0, 1)$. By Dirichlet's approximation theorem, there exists $(b, r)=1$ with $0\leq b\leq r$ and $1\leq r\leq H$ such that
\begin{equation}\label{DirichletApprox}
\left|\alpha-\frac{b}{r}\right|\leq \frac{1}{rH}.
\end{equation}
Let $R=\log H$. We say that $\alpha$ lies in a ``major arc'' if such an approximation exists with $r\leq R$, and otherwise $\alpha$ is said to lie in a ``minor arc''.

We first prove \eqref{BoundSumgh} when $\alpha$ lies in a minor arc. In this case we have $||r\alpha||>1/H$ for all $1\leq r\leq R$. Thus, it follows from Lemma \ref{MinorMajor} that
$$  \sum_{h\leq H} \frac{g(2\pi\alpha h)}{h} = \frac{a_0}{2} \log H -\sum_{1\leq r\leq R}a_r \log |1-e(r\alpha )| +O\left(1 + \sum_{1\leq r\leq R}|a_r|\right).
$$
Moreover, since $a_r\leq 0$ and $|a_r|\ll 1/r^2$ for all $r\ge 1$,  and $\log |1-e(r\alpha)|\leq \log 2$, we deduce that 
$$ \sum_{h\leq H} \frac{g(2\pi\alpha h)}{h} = \frac{a_0}{2} \log H +\sum_{1\leq r\leq R}|a_r| \log |1-e(r\alpha )|+O(1)\leq  \frac{a_0}{2} \log H +O(1),$$
which yields the result in this case.

We now suppose that $\alpha$ lies in a major arc. In this case there exists a rational number $b/r$ such that $(b, r)=1$, $1\leq r\leq R$ and $|\alpha-b/r|\leq 1/rH$. Since $g$ is continuous and has a piecewise continuous derivative, we have $|g(2\pi\alpha h)- g(2\pi bh/r)|\ll h/H.$ Therefore, appealing to Lemma \ref{MinorMajor} we obtain
\begin{align*}
\sum_{h\leq H} \frac{g(2\pi\alpha h)}{h} 
&= \sum_{h\leq H}\frac{g\left(\frac{2\pi bh}{r}\right)}{h}+O(1)\\
&=\left(\frac{a_0}{2}  +\sum_{1\leq m\leq R/r} a_{mr}\right) \log H- \sum_{\substack{1\leq n\leq R\\ r\nmid n}}a_n \log |1-e(nb/r)| +O(1).
\end{align*}
The inequality \eqref{BoundSumgh} follows in this case upon noting that $a_r\leq 0$ and $|a_r|\ll 1/r^2$ for all $r\ge 1$, and $\log |1-e(nb/r)|\leq \log 2$. This completes the proof.
\end{proof}

\section{Moments of sums of Birch sums and ingredients from algebraic geometry}\label{Moments}

In this section we shall investigate the $2k$-th moment of sums of Birch sums
\begin{equation}\label{ExpBi}
\sum_{y\le |h|<z} c(h) \B(a-h),
\end{equation}
where $1\le y<z<p/2$ are real numbers and $\{c(h)\}_{h\in \mathbb{Z}}$ is a sequence of complex numbers. For $k$ fixed, these moments were computed by Kowalski and Sawin \cite{KoSa} using deep tools from algebraic geometry, namely Deligne's equidistribution theorem, the Goursat-Kolchin-Ribet criterion of Katz, as well as Katz's computations for the monodromy groups of a certain sheaf attached to the exponential sums $\B(a)$ (see \cite{Ka}). However, in our case we need asymptotic formulas for these moments that hold uniformly in the region $k\leq (\log p)^{1-o(1)}$. To this end, we shall use a uniform version of Lemma 2.5 of \cite{KoSa}, which we extract from the recent work of Perret-Gentil \cite{Pe1} on $\ell$-adic trace functions over finite fields.
\begin{lem}\label{Algebraic}
Let $p>7$ be prime. For all positive integers $1\leq k\leq (\log p)/2$, and all $h_1, \dots, h_k\in \F$ we have 
$$\frac{1}{p-1}\sum_{a\in \Fs}\B(a-h_1) \cdots \B(a-h_k) = \ex\big(\X(h_1)\cdots  \X(h_k)\big) + O\left(\frac{2^k k}{\sqrt{p}}\right),$$
where $\{\X(h)\}_{h\in \mathbb{Z}}$ is a sequence of independent random variables with Sato-Tate distributions on $[-2, 2]$, and the implied constant is absolute. 
\end{lem}
\begin{proof}
First, we write 
$$\sum_{a\in \Fs}\B(a-h_1) \cdots \B(a-h_k) = \sum_{a\in \Fs}\B(a-j_1)^{b_1} \cdots \B(a-j_{m}) ^{b_{m}},$$
where $j_1, \dots, j_{m}$ are distinct, and $b_1+\cdots + b_{m}=k$.

Let $\mathcal{S}$ be the rank $2$ lisse sheaf on $\mathbb{A}^1_{\F}$ parameterizing the Birch sums $\B(a)$ (see Katz \cite{Ka} for a reference on these sheaves and their monodromy groups). Katz (see Th. 19 and Cor. 20 of \cite{Ka}) showed that the geometric and arithmetic monodromy groups of the sheaf $\mathcal{S}$ are both equal to $\textup{SL}_2$ for $p > 7$. Furthermore, it follows from the discussion in the beginning of page 15 of \cite{KoSa} that for $\tau\neq 0$, there is no geometric isomorphism
$$ [+\tau]^*\mathcal{S} \simeq \mathcal{S} \otimes \mathcal{L}, $$
where $\mathcal{L}$ is a rank $1$ sheaf over $\F$. Thus, we can apply Proposition 2.4 of \cite{Pe1} which gives
\begin{equation}\label{PER}
\frac{1}{p-1} \sum_{a\in \Fs}\B(a-j_1)^{b_1} \cdots \B(a-j_{m}) ^{b_{m}} = \prod_{i=1}^m \textup{mult}_1 (\textup{Std}^{\otimes b_i})+O\left(\frac{2^{k} k}{\sqrt{p}}\right),
\end{equation}
where $\textup{mult}_1(\textup{Std}^{\otimes b})$ is the multiplicity of the trivial representation of $\textup{SU}_2$ in the $b$-th tensor power of its standard $2$-dimensional representation. Finally, it follows from the representation theory of $\textup{SU}_2$ that 
$$  \textup{mult}_1 (\textup{Std}^{\otimes b})= \ex(\Y^b),$$
for any random-variable $\Y$ with Sato-Tate distribution $\mu_{ST}$. Thus, we deduce from \eqref{PER} that 
$$ \frac{1}{p-1} \sum_{a\in \Fs}\B(a-j_1)^{b_1} \cdots \B(a-j_{m}) ^{b_{m}}= \ex\left(\X(j_1)^{b_1}\cdots  \X(j_m)^{b_m}\right) +O\left(\frac{2^{k} k}{\sqrt{p}}\right), 
$$
where $\X(j_1), \dots, \X(j_m)$ are independent random variables with Sato-Tate distributions on $[-2, 2]$. This completes the proof.
\end{proof}
Using this result we prove the following proposition.
\begin{pro}\label{AsympMomentsyz}
Let $\{c(h)\}_{h\in \mathbb{Z}}$ be a sequence of complex numbers, and $p$ be a large prime. Let $0\leq y<z\leq p/2$ be real numbers and $k, \ell \leq (\log p)/4$ be positive integers. Then, we have  
\begin{align*}
&\frac{1}{p-1}\sum_{a\in \Fs}\left(\sum_{y\le |h|<z} c(h) \B(a-h)\right)^{k}\left(\sum_{y\le |h|<z} \overline{c(h)} \B(a-h)\right)^{\ell}\\
 &= \ex\left(\left(\sum_{y\le |h|<z} c(h) \X(h)\right)^{k}\left(\sum_{y\le |h|<z} \overline{c(h)} \X(h)\right)^{\ell}\right) +O\left(p^{-1/2}\Big(4\sum_{y\le |h|<z} |c(h)|\Big)^{k+\ell} \right),
\end{align*}
where $\{\X(h)\}_{|h|\geq 1} $ is a sequence of independent random variables with Sato-Tate distributions on $[-2, 2]$. 
\end{pro}
\begin{proof} 
It follows from Lemma \ref{Algebraic} that
\begin{align*}
&\frac{1}{p-1}\sum_{a\in \Fs}\left(\sum_{y\le |h|<z} c(h) \B(a-h)\right)^{k}\left(\sum_{y\le |h|<z} \overline{c(h)} \B(a-h)\right)^{\ell}\\
&= \sum_{\substack{y\le |h_1|, \dots, |h_k|<z\\ y\le |r_1|, \dots, |r_{\ell}|<z}} c(h_1)\cdots c(h_k) \overline{c(r_1)\cdots c(r_{\ell})} \ \frac{1}{p-1}\sum_{a\in \Fs}\prod_{u=1}^k\B(a-h_u)\prod_{v=1}^{\ell}\B(a-r_v)\\
&=\sum_{\substack{y\le |h_1|, \dots, |h_k|<z\\ y\le |r_1|, \dots, |r_{\ell}|<z}} c(h_1)\cdots c(h_k) \overline{c(r_1)\cdots c(r_{\ell})} \ \ex\left(\prod_{u=1}^k \X(h_{u})\prod_{v=1}^{\ell}\X(r_v)\right) + E_{k, \ell}(y, z),
\end{align*}
where the error term satisfies
$$ E_{k, \ell}(y, z) \ll \frac{2^{k+\ell}(k+\ell)}{\sqrt{p}} \Big(\sum_{y\leq |h|<z} |c(h)|\Big)^{k+\ell} \ll p^{-1/2}\Big(4\sum_{y\le |h|<z} |c(h)|\Big)^{k+\ell}. $$
The result follows upon noting that
\begin{equation}\label{Random3}
\begin{aligned}
&\sum_{\substack{y\le |h_1|, \dots, |h_k|<z\\ y\le |r_1|, \dots, |r_{\ell}|<z}} c(h_1)\cdots c(h_k) \overline{c(r_1)\cdots c(r_{\ell})} \ \ex\left(\prod_{u=1}^k \X(h_{u})\prod_{v=1}^{\ell}\X(r_v)\right)\\
&=\ex\left(\left(\sum_{y\le |h|<z} c(h) \X(h)\right)^{k}\left(\sum_{y\le |h|<z} \overline{c(h)} \X(h)\right)^{\ell}\right).
\end{aligned}
\end{equation}
\end{proof}
Next, we prove uniform bounds for the moments of the sum of random variables $\sum_{y\le |h|<z} c(h) \X(h)$, where $c(h)$ are complex numbers that satisfy $c(h)\ll1/|h|$ for $|h|\geq 1$. These bounds will be used in the proofs of the lower and upper bounds of Theorem \ref{Main}.
\begin{lem}\label{BOUNDRAND}
Let $\{c(h)\}_{h\in \mathbb{Z}}$ be a sequence of complex numbers such that $|c(h)|\leq c_0/|h|$ for $|h|\geq 1$, where $c_0$ is a positive constant. Let $1\leq y<z$ be real numbers. Then, for all positive integers $k$ we have 
\begin{equation}\label{LargeY}
\ex\left(\left|\sum_{y\le |h|<z} c(h) \X(h)\right|^k\right) \leq \left(\frac{8c_0^2k}{y}\right)^{k/2}.
\end{equation}
Moreover, if $k>y$ then 
\begin{equation}\label{SmallY}
\ex\left(\left|\sum_{y\le |h|<z} c(h) \X(h)\right|^k\right) \leq (15c_0\log k)^k.
\end{equation}
\end{lem}
Combining this result with Proposition \ref{AsympMomentsyz} we deduce uniform bounds for the moments of the sums \eqref{ExpBi}.
\begin{cor}\label{BoundMomentsyz}
Let $\{c(h)\}_{h\in \mathbb{Z}}$ be a sequence of complex numbers such that $|c(h)|\leq c_0/|h|$ for $|h|\geq 1$, where $c_0$ is a positive constant. Let $p$ be a large prime and $1\leq y <z\leq p/2$ be real numbers. Then, for all positive integers  $ k\leq (\log p)/(5\log\log p)$ we have   
$$
\frac{1}{p-1}\sum_{a\in \Fs}\left|\sum_{y\le |h|<z} c(h) \B(a-h)\right|^{2k} \ll \left(\frac{16c_0^2k}{y}\right)^{k}+\frac{(16c_0\log p)^{2k}}{p^{1/2}}.
$$
\end{cor}
\begin{proof} It follows from Proposition \ref{AsympMomentsyz} that
\begin{align*}
\frac{1}{p-1}\sum_{a\in \Fs}\left|\sum_{y\le |h|<z} c(h) \B(a-h)\right|^{2k} 
&=\ex\left(\left|\sum_{y\le |h|<z} c(h) \X(h)\right|^{2k}\right)\\
& \ \ \  +O\left(\frac{(16c_0\log p)^{2k}}{p^{1/2}}\right), 
\end{align*}
since $\sum_{|h| < z} |c(h)|\leq 4c_0 \log p$. Using \eqref{LargeY} completes the proof.
\end{proof}

\begin{proof}[Proof of Lemma \ref{BOUNDRAND}]
We first prove \eqref{LargeY} when $k=2m$ is even. 
By \eqref{Random3} we have
\begin{equation}\label{b1}
\ex\left(\left|\sum_{y\le |h|<z} c(h) \X(h)\right|^{2m}\right)\leq 
c_0^{2m}\sum_{y\le |h_1|, \dots, |h_{2m}|<z} \frac{\left|\ex\left(\X(h_1) \cdots \X(h_{2m})\right)\right|}{|h_{1}\cdots h_{2m}|}. 
\end{equation}
Recall that if $\X$ is a random variable with Sato-Tate distribution $\mu_{ST}$ and $\ell$ is a positive integer then 
$$ \ex(\X^{\ell})= \begin{cases}  \frac{1}{n+1} \binom{2n}{n}& \text{ if } \ell=2n \text{ is even}, \\ 0 &  \text{ if } \ell \text{ is odd}. \end{cases}
$$
Hence, we obtain
\begin{equation}\label{b2}
\begin{aligned}
&\sum_{y\le |h_1|, \dots, |h_{2m}|<z} \frac{\left|\ex\left(\X(h_1) \cdots \X(h_{2m})\right)\right|}{|h_{1}\cdots h_{2m}|}\\
 &= \sum_{\substack{h_1<\cdots <h_{\ell}\\ y\le |h_1|, \dots, |h_{\ell}|<z}}\sum_{\substack{n_1, \dots, n_{\ell}\geq 1\\ n_1+\cdots +n_{\ell}=2m}} \binom{2m}{n_1, \dots, n_{\ell}} \frac{|\ex(\X(h_1)^{n_1})| \cdots |\ex(\X(h_{\ell})^{n_{\ell}})|}{|h_{1}^{n_1}\cdots h_{\ell}^{n_{\ell}}|}\\
&\leq  \sum_{\substack{h_1<\cdots <h_{\ell}\\ y\le |h_1|, \dots, |h_{\ell}|<z}}\sum_{\substack{r_1, \dots, r_{\ell}\geq 1\\ r_1+\cdots +r_{\ell}=m}} \binom{2m}{2r_1, \dots, 2r_{\ell}} \frac{\binom{2r_1}{r_1}\cdots\binom{2r_{\ell}}{r_{\ell}}}{h_{1}^{2r_1}\cdots h_{\ell}^{2r_{\ell}}}\\
&\leq \frac{(2m)!}{m!} \left(\sum_{y\leq |h|< z} \frac{1}{h^2}\right)^m,
\end{aligned}
\end{equation}
since 
$$\binom{2m}{2r_1, \dots, 2r_{\ell}} \binom{2r_1}{r_1}\cdots\binom{2r_{\ell}}{r_{\ell}} \leq \frac{(2m)!}{m!} \binom{m}{r_1, \dots, r_{\ell}}.$$
Thus, combining the estimates \eqref{b1} and \eqref{b2}, together with the elementary inequalities $(2m)!/m!\leq (2m)^m$ and 
$\sum_{y\leq |h|< z} 1/h^2\leq 4/y$  we obtain 
\begin{equation}\label{EVENBound}
\ex\left(\left|\sum_{y\le |h|<z} c(h) \X(h)\right|^{2m}\right) \leq \left(\frac{8c_0^2m}{y}\right)^{m}.
\end{equation}
We now establish \eqref{LargeY} when $k$ is odd. By the Cauchy-Schwarz inequality and \eqref{EVENBound} we have 
\begin{align*} 
\ex\left(\left|\sum_{y\le |h|<z} c(h) \X(h)\right|^{k}\right)
& \leq \ex\left(\left|\sum_{y\le |h|<z} c(h) \X(h)\right|^{2k-2}\right)^{1/2}\ex\left(\left|\sum_{y\le |h|<z} c(h) \X(h)\right|^{2}\right)^{1/2}\\
& \leq \left(\frac{8c_0^2k}{y}\right)^{k/2}, 
\end{align*}
as desired. 

We now prove \eqref{SmallY}. By \eqref{LargeY} and Minkowski's inequality we have 
\begin{align*}
\ex\left(\left|\sum_{y\le |h|<z} c(h) \X(h)\right|^k\right)^{1/k}
&\leq \ex\left(\left|\sum_{y\le |h| < k} c(h) \X(h)\right|^k\right)^{1/k} +\ex\left(\left|\sum_{k\le |h|<z} c(h) \X(h)\right|^k\right)^{1/k}\\
&\leq 2c_0 \sum_{y\le |h|<k} \frac{1}{|h|}+ \sqrt{8} c_0.\\
&\leq 15c_0\log k .\\
\end{align*}
This completes the proof.

\end{proof}

\section{Completing the proof of the upper bound in Theorem \ref{Main}: Proof of Theorem \ref{MomentsMaxTail}}
Let $p$ be a large prime, and $k$ be a large positive integer such that $k\leq (\log p)/(100\log\log p)$. Let $y\leq k^2$ be a positive real number. Then, it follows from
Minkowski's inequality that
\begin{align*}
&\left(\sum_{a\in \Fs} \max_{\alpha\in S}\left|\sum_{y\le |h|<p/2} \frac{e\left(\alpha h\right)- 1}{h} \B(a-h)\right|^{2k}\right)^{1/2k}\\
 \leq 
 &  \ \left(\sum_{a\in \Fs} \max_{\alpha\in S}\left|\sum_{y\le |h|<k^2} \frac{e\left(\alpha h\right)- 1}{h} \B(a-h)\right|^{2k}\right)^{1/2k}\\ 
& \ \ \ \ \ \ \ \ \ \ \ \ \ \ \ \ \ \ \ \ \ \ \ \ \ \ \ \ \ \ +\left(\sum_{a\in \Fs} \max_{\alpha\in S}\left|\sum_{k^2\le |h|<p/2} \frac{e\left(\alpha h\right)- 1}{h} \B(a-h)\right|^{2k}\right)^{1/2k}.
\end{align*}
Therefore, Theorem \ref{MomentsMaxTail} is an immediate consequence of the following propositions.
\begin{pro}\label{MomentsMaxTail1}
Let $p$ be a large prime, and $k$ be a large positive integer such that $k\leq (\log p)/(100\log\log p)$. Let $S$ be a non-empty subset of $[0, 1)$, and put $y=10^5 k$.  Then we have 
$$ 
\frac{1}{p-1}\sum_{a\in \Fs} \max_{\alpha\in S}\left|\sum_{y\le |h|<k^2} \frac{e\left(\alpha h\right)- 1}{h} \B(a-h)\right|^{2k} \ll e^{-4k}.
$$
\end{pro}
\begin{pro}\label{MomentsMaxTail2}
Let $p$ be a large prime, and $k$ be a large positive integer such that $k\leq (\log p)/(100\log\log p)$. Let $S$ be a non-empty subset of $[0, 1)$ such that $|S|\leq \sqrt{p}$. Then we have $$ 
\frac{1}{p-1}\sum_{a\in \Fs} \max_{\alpha\in S}\left|\sum_{k^2 \le |h|<p/2} \frac{e\left(\alpha h\right)- 1}{h} \B(a-h)\right|^{2k} \ll e^{-4k} + \frac{|S|(4\log p)^{8k}}{\sqrt{p}}.
$$
\end{pro}
We start by proving Proposition \ref{MomentsMaxTail1}, since its proof is simpler due to the fact that the inner sum over $|h|$ is very short.
\begin{proof}[Proof of  Proposition \ref{MomentsMaxTail1}] 
First, if $|S|\leq k^4$ then by Corollary \ref{BoundMomentsyz} we have 
\begin{equation}\label{FirstCaseBound}
\begin{aligned}
\sum_{a\in \Fs} \max_{\alpha\in S}\left|\sum_{y\le |h|<k^2} \frac{e\left(\alpha h\right)- 1}{h} \B(a-h)\right|^{2k}
&\leq \sum_{\alpha\in S}\sum_{a\in \Fs}\left|\sum_{y\le |h|<k^2} \frac{e\left(\alpha h\right)- 1}{h} \B(a-h)\right|^{2k} \\
&\ll k^4 (p-1) \left(\left(\frac{64 k}{y}\right)^{k}+\frac{(32\log p)^{2k}}{\sqrt{p}}\right)\\
& \ll e^{-4 k} (p-1).
\end{aligned}
\end{equation}
We now suppose that $|S|>k^4$ and define
$\mathcal{A}_k= \{b/k^4 : 1\leq b\leq k^4\}.$ Then for all $\alpha \in S$, there exists $\beta_{\alpha}\in \mathcal{A}_k$ such that 
$|\alpha-\beta_{\alpha}|\leq 1/k^4$. In this case we have $e(\alpha h)= e(\beta_{\alpha} h)+ O(h/k^4),$
 and hence 
$$\sum_{y\le |h|<k^2} \frac{e\left(\alpha h\right)- 1}{h} \B(a-h)= \sum_{y\le |h|<k^2} \frac{e(\beta_{\alpha} h)- 1}{h} \B(a-h) +O\left(\frac{1}{k^2}\right). 
 $$
 Therefore, using the simple inequality $|x+y|^{2k} \leq (2\max(|x|, |y|))^{2k}\leq 2^{2k} (|x|^{2k}+|y|^{2k})$ we deduce that
 \begin{equation}\label{Transition}
\begin{aligned}
&\max_{\alpha\in S}\left|\sum_{y\le |h|<k^2} \frac{e\left(\alpha h\right)- 1}{h} \B(a-h)\right|^{2k}\leq 2^{2k} \max_{\alpha\in \mathcal{A}_k}\left|\sum_{y\le |h|<k^2} \frac{e\left(\alpha h\right)- 1}{h} \B(a-h)\right|^{2k}+ \left(\frac{c_1}{k^2}\right)^{2k}
\\
& \leq 2^{2k} \sum_{\alpha\in \mathcal{A}_k}\left|\sum_{y\le |h|<k^2} \frac{e\left(\alpha h\right)- 1}{h} \B(a-h)\right|^{2k}+ \left(\frac{c_1}{k^2}\right)^{2k},
 \end{aligned}
 \end{equation}
 for some positive constant $c_1$. Thus, it follows from Corollary \ref{BoundMomentsyz}  that in this case we have
 \begin{equation}\label{SecondCaseBound}
\begin{aligned}
&\frac{1}{p-1}\sum_{a\in \Fs} \max_{\alpha\in S}\left|\sum_{y\le |h|<k^2} \frac{e\left(\alpha h\right)- 1}{h} \B(a-h)\right|^{2k}\\
& \leq 2^{2k} \sum_{\alpha\in \mathcal{A}_k} \frac{1}{p-1} \sum_{a\in \Fs}\left|\sum_{y\le |h|<k^2} \frac{e\left(\alpha h\right)- 1}{h} \B(a-h)\right|^{2k} + \left(\frac{c_1}{k^2}\right)^{2k}\\
&\ll k^4 2^{2k}\left(\left(\frac{64 k}{y}\right)^{k}+\frac{(32\log p)^{2k}}{\sqrt{p}}\right)+\left(\frac{c_1}{k^2}\right)^{2k}\ll e^{-4k},
\end{aligned}
\end{equation}
which completes the proof.
\end{proof}
\begin{proof}[Proof of Proposition \ref{MomentsMaxTail2}]
Since the inner sum over $h$ is very long in this case, we shall split it into dyadic intervals. Let $J_1= \lfloor \log(k^2)/\log 2\rfloor$ and $J_2= \lfloor \log(p/2)/\log 2\rfloor$. We define $z_{J_1}:= k^2$, $z_{J_2+1}:=p/2$, and $z_j:=2^{j}$ for $J_1+1\leq j\leq J_2$. Then, using H\"older's inequality we obtain
\begin{equation}\label{DyadicSums}
\begin{aligned}
& \left|\sum_{k^2\le |h|<p/2} \frac{e\left(\alpha h\right)- 1}{h} \B(a-h)\right|^{2k}= \left|\sum_{J_1\leq j\leq J_2}\frac{1}{j^2} \cdot \left(j^2\sum_{z_j\le |h|<z_{j+1}} \frac{e\left(\alpha h\right)- 1}{h} \B(a-h)\right)\right|^{2k}\\
& \ \ \ \ \ \ \ \leq \left(\sum_{J_1\leq j\leq J_2} \frac{1}{j^{4k/(2k-1)}}\right)^{2k-1} \left(\sum_{J_1\leq j\leq J_2} j^{4k} \left|\sum_{z_j\le |h|<z_{j+1}} \frac{e\left(\alpha h\right)- 1}{h} \B(a-h)\right|^{2k}\right)\\
&  \ \ \ \ \ \ \ \leq \left(\frac{c_2}{\log k}\right)^{2k+1} \sum_{J_1\leq j\leq J_2} j^{4k} \left|\sum_{z_j\le |h|<z_{j+1}} \frac{e\left(\alpha h\right)- 1}{h} \B(a-h)\right|^{2k},
\end{aligned}
\end{equation}
for some constant $c_2>0$. Therefore, this reduces the problem to bounding the corresponding moments over each dyadic interval $[z_j, z_{j+1}]$, namely
$$\frac{1}{p-1}\sum_{a\in \Fs} \max_{\alpha\in S}\left|\sum_{z_j\le |h|<z_{j+1}} \frac{e\left(\alpha h\right)- 1}{h} \B(a-h)\right|^{2k}.
$$
We shall consider two cases, depending on whether $j$ is large in terms of $|S|$. First, if $4^j\geq  |S|$ then by Corollary \ref{BoundMomentsyz} we have
\begin{equation}\label{Largej}
\begin{aligned}
&\frac{1}{p-1}\sum_{a\in \Fs} \max_{\alpha\in S}\left|\sum_{z_j\le |h|<z_{j+1}} \frac{e\left(\alpha h\right)- 1}{h} \B(a-h)\right|^{2k}\\
& \leq  \sum_{\alpha\in S}\frac{1}{p-1}\sum_{a\in \Fs} \left|\sum_{z_j\le |h|<z_{j+1}} \frac{e\left(\alpha h\right)- 1}{h} \B(a-h)\right|^{2k} \ll 4^j\left(\frac{64k}{2^j}\right)^{k}+\frac{|S|(32\log p)^{2k}}{\sqrt{p}}.
\end{aligned}
\end{equation}
since $z_j \geq 2^j$ for $J_1\leq j\leq J_2$.  We now suppose that $4^j<|S|$, and let $\mathcal{B}_j= \{b/4^j : 1\leq b\leq 4^j\}$. Then for all $\alpha \in S$ there exists $\beta_{\alpha}\in \mathcal{B}_j$ such that 
$|\alpha-\beta_{\alpha}|\leq 1/4^j$. In this case we have $e(\alpha h)= e(\beta_{\alpha} h)+ O(h/4^j),$
 and hence we obtain
 $$ \sum_{z_j\le |h|<z_{j+1}} \frac{e\left(\alpha h\right)- 1}{h} \B(a-h)= \sum_{z_j\le |h|<z_{j+1}} \frac{e(\beta_{\alpha} h)- 1}{h} \B(a-h) +O\left(\frac{1}{2^j}\right),
 $$
 since $z_{j+1}\asymp z_j\asymp 2^j$.  Therefore, similarly to \eqref{Transition} we derive
\begin{align*}
&\max_{\alpha\in S}\left|\sum_{z_j\le |h|<z_{j+1}} \frac{e\left(\alpha h\right)- 1}{h} \B(a-h)\right|^{2k}\\
& \leq 2^{2k} \max_{\alpha\in \mathcal{B}_j}\left|\sum_{z_j\le |h|<z_{j+1}} \frac{e\left(\alpha h\right)- 1}{h} \B(a-h)\right|^{2k}+ \left(\frac{c_3}{2^j}\right)^{2k},
 \end{align*}
 for some positive constant $c_3$. Thus, appealing to Corollary \ref{BoundMomentsyz} we get
 \begin{equation}\label{Smallj}
 \begin{aligned}
&\frac{1}{p-1}\sum_{a\in \Fs} \max_{\alpha\in S}\left|\sum_{z_j\le |h|<z_{j+1}} \frac{e\left(\alpha h\right)- 1}{h} \B(a-h)\right|^{2k} \\
&\leq 2^{2k}\sum_{\alpha\in \mathcal{B}_j} \frac{1}{p-1}\sum_{a\in \Fs}\left|\sum_{z_j\le |h|<z_{j+1}} \frac{e\left(\alpha h\right)- 1}{h} \B(a-h)\right|^{2k}+ \left(\frac{c_3}{2^j}\right)^{2k}\\
& \ll 4^j\left(\frac{2^8k}{2^j}\right)^{k}+\frac{|S|(64\log p)^{2k}}{\sqrt{p}},
 \end{aligned}
 \end{equation}
 since $|\mathcal{B}_j|=4^j <|S|.$ Combining \eqref{Largej} and \eqref{Smallj} we deduce that in all cases we have
 $$
 \frac{1}{p-1}\sum_{a\in \Fs} \max_{\alpha\in S}\left|\sum_{z_j\le |h|<z_{j+1}} \frac{e\left(\alpha h\right)- 1}{h} \B(a-h)\right|^{2k}  \ll 4^j\left(\frac{2^8k}{2^j}\right)^{k}+\frac{|S|(64\log p)^{2k}}{\sqrt{p}}.
 $$
Inserting this bound in \eqref{DyadicSums} gives
\begin{equation}\label{ThirdCaseBound}
\begin{aligned}
&\frac{1}{p-1}\sum_{a\in \Fs} \max_{\alpha\in S}\left|\sum_{k^2\le |h|<p/2} \frac{e\left(\alpha h\right)- 1}{h} \B(a-h)\right|^{2k}\\
 &\ll \left(\frac{c_4}{\log k}\right)^{2k+1} k^k \sum_{J_1\leq j\leq J_2} 4^j\left(\frac{j^4}{2^{j}}\right)^{k}+ \frac{|S|(4\log p)^{8k}}{\sqrt{p}}\\
 &\ll e^{-4k} + \frac{|S|(4\log p)^{8k}}{\sqrt{p}},
\end{aligned}
\end{equation}
for some positive constant $c_4$, since $j^4\leq 2^{j/4}$ for $j$ large enough, and $2^{J_1}\asymp k^2$. This completes the proof.
\end{proof}

\section{Proof of the lower bound of Theorem \ref{Main}}
To prove the lower bound of Theorem \ref{Main} we shall investigate the sum \eqref{PS} in the special case $x=p/2$. By   \eqref{Planche} we have 
\begin{equation}\label{EqualityB}
\frac{1}{\sqrt{p}} \sum_{0\leq n\leq p/2} e_p(n^3+an)=\sum_{|h|<p/2}\gamma_p(h) \B(a-h),
\end{equation}
where $\gamma_p(0)=1/2+1/(2p)$, and for $|h|\geq 1$ we have
$$\gamma_p(h)= \frac{1}{p} \sum_{0\leq m\leq p/2}e_p(mh)=  \frac{e_p\left(h(p+1)/2\right)-1}{p\left(e_p\left(h\right)-1\right)}. $$
Moreover, for $1\leq |h|<p/2$ we have
\begin{equation}\label{BGM}
|\gamma_p(h)|\leq \frac{1}{p|\sin(\pi h/p)|} \leq \frac{1}{2|h|},
\end{equation}
since $\sin(\pi \alpha)\geq 2\alpha$ for $0\leq \alpha\leq 1/2$.
Furthermore, we also have
\begin{equation}\label{OddEvenGM}
\gamma_p(h)=\frac{e^{\pi i h}-1}{2\pi i h} +O\left(\frac{1}{p}\right)= \begin{cases} O(\frac1p) & \text{ if } h \text{ is even},\\ \frac{i}{\pi h}+ O(\frac1p) & \text{ if } h \text{ is odd}.
\end{cases}
\end{equation}
We shall prove the following theorem from which the lower bound of Theorem \ref{Main} follows.
\begin{thm}\label{LowerB}
Let $p$ be a large prime. Uniformly for $V$ in the range $1\ll V\leq \frac{2}{\pi} \log\log p- 2\log\log\log p$ we have
\begin{align*}
& \frac{1}{p-1} \left| \left \{a\in \Fs: \frac{1}{\sqrt{p}} \im\sum_{0\leq n\leq p/2} e_p(n^3+an)> V\right \}\right| \\
& = \exp\left(-A_0 \exp\left(\frac{\pi}{2} V\right) \left(1+O\left(\sqrt{V} e^{-\pi V/4}\right)\right)\right).
\end{align*}
Furthermore, the same estimate holds for the proportion of $a\in \F$ such that

\noindent  $\frac{1}{\sqrt{p}}\im\sum_{0\leq n\leq p/2} e_p(n^3+an)<- V$, in the same range of $V$.
\end{thm}

Here and throughout we let 
$$ \widetilde{\gamma_p}(h) = \im \left(\gamma_p(h)\right).$$
The first step in the proof of Theorem \ref{LowerB} is to show that the Laplace transform of the sum $\sum_{|h|<p/2}\widetilde{\gamma_p}(h) \B(a-h)$ (after removing a small set of ``bad'' points $a$) is very close to that of the probabilistic random model $\sum_{|h|<p/2}\widetilde{\gamma_p}(h) \X(h)$.
\begin{pro}\label{Laplace}
Let $p$ be a large prime. There exists a set $\mathcal{E}_p\subset \Fs$ with cardinality $|\mathcal{E}_p|\leq p^{9/10}$ such that for all complex numbers $s$ with $|s|\leq (\log p)/(50\log \log p)^2$ we have
\begin{align*}
 \frac{1}{p-1} \sum_{a\in \Fs\setminus \mathcal{E}_p} \exp\left(s \cdot  \sum_{|h|<p/2}\widetilde{\gamma_p}(h) \B(a-h)\right)
 &= \ex\left(\exp\left(s \cdot \sum_{|h|<p/2}\widetilde{\gamma_p}(h) \X(h)\right)\right)\\
 & \ \ \  + O\left(\exp\left(-\frac{\log p}{100\log\log p}\right)\right).
\end{align*}
\end{pro}
\begin{proof}
Let $\mathcal{E}_p$ be the set of $a\in \Fs$ such that 
$$ \left|\sum_{|h|<p/2}\widetilde{\gamma_p}(h) \B(a-h)\right| \geq 6 \log\log p.$$
Using the bounds $|\widetilde{\gamma_p}(h)|\leq 1/(2|h|)$ for $|h|\geq 1$, and $|\B(a-h)|\leq 2$ we get 
\begin{align*}
\left|\sum_{|h|<p/2}\widetilde{\gamma_p}(h) \B(a-h)\right|
& \leq \sum_{1\leq |h|<(\log p)^2} \frac{1}{|h|} + \left|\sum_{(\log p)^2<|h|<p/2}\widetilde{\gamma_p}(h) \B(a-h)\right|\\
& \leq 5 \log\log p + \left|\sum_{(\log p)^2<|h|<p/2}\widetilde{\gamma_p}(h) \B(a-h)\right|,
\end{align*}
if $p$ is sufficiently large. Therefore, it follows from Corollary \ref{BoundMomentsyz} that for $\ell=\lfloor \log p/(10\log\log p)\rfloor$ we have
\begin{equation}\label{BoundE}
\begin{aligned}
|\mathcal{E}_p| 
&\leq \Big|\Big\{a\in \Fs: \big|\sum_{(\log p)^2<|h|<p/2}\widetilde{\gamma_p}(h) \B(a-h)\big|\geq \log\log p\Big\}\Big|\\
& \leq (\log\log p)^{-2\ell} \sum_{a\in \Fs} \left|\sum_{(\log p)^2<|h|<p/2}\widetilde{\gamma_p}(h) \B(a-h)\right|^{2\ell} \\
& \ll p^{9/10}.
\end{aligned}
\end{equation}

Let $N=\lfloor \log p/(20\log\log p)\rfloor$. Then we have
\begin{equation}\label{TaylorLaplace}
\begin{aligned}
&\frac{1}{p-1}\sum_{a\in \Fs\setminus \mathcal{E}_p}\exp\left(s \cdot  \sum_{|h|<p/2}\widetilde{\gamma_p}(h) \B(a-h)\right)\\
&= \sum_{k=0}^N \frac{s^k}{k!} \frac{1}{p-1}\sum_{a\in \Fs\setminus \mathcal{E}_p} \left(\sum_{|h|<p/2}\widetilde{\gamma_p}(h) \B(a-h)\right)^k+E_1\\
\end{aligned}
\end{equation}
where
$$ E_1\ll \sum_{k>N} \frac{|s|^k}{k!} (6\log\log p)^k\leq \sum_{k>N} \left(\frac{20|s|\log\log p}{N}\right)^k \ll e^{-N} $$
by Stirling's formula and our assumption on $s$. Furthermore, note that 
$$ \sum_{|h|<p/2}|\widetilde{\gamma_p}(h) \B(a-h)| \leq \sum_{1\leq |h|< p/2} \frac{1}{|h|}\leq 5 \log p.$$ 
Therefore, it follows from Proposition \ref{AsympMomentsyz} and equation \eqref{BoundE} that for all integers $0\leq k\leq N$ we have 
\begin{align*}
&\frac{1}{p-1}\sum_{a\in \Fs\setminus \mathcal{E}_p} \left(\sum_{|h|<p/2}\widetilde{\gamma_p}(h) \B(a-h)\right)^k\\
&=\frac{1}{p-1}\sum_{a\in \Fs} \left( \sum_{|h|<p/2}\widetilde{\gamma_p}(h) \B(a-h)\right)^k +O\left(p^{-1/10}(5\log p)^{k}\right)\\
&=\ex\left(\Big( \sum_{|h|<p/2}\widetilde{\gamma_p}(h) \X(h)\Big)^k\right) +O\left(p^{-1/25}\right).
\end{align*}
Moreover, it follows from Lemma \ref{BOUNDRAND} and Stirling's formula that  
$$ \sum_{k>N} \frac{|s|^k}{k!} \ex\left(\Big|\sum_{|h|<p/2}\widetilde{\gamma_p}(h) \X(h)\Big|^k\right) \ll \sum_{k>N} \left(\frac{30|s|\log k}{k}\right)^k\ll \sum_{k>N} \left(\frac{30|s|\log N}{N}\right)^k\ll e^{-N},$$
Finally, inserting these estimates in \eqref{TaylorLaplace},  we derive
\begin{align*}
&\frac{1}{p-1}\sum_{a\in \Fs\setminus \mathcal{E}_p}\exp\left(s \cdot \sum_{|h|<p/2}\widetilde{\gamma_p}(h) \B(a-h)\right)\\
&=  \sum_{k=0}^N \frac{s^k}{k!} \ex\left(\Big(\sum_{|h|<p/2}\widetilde{\gamma_p}(h) \X(h)\Big)^k\right) +O\left(e^{-N}+ p^{-1/20}e^{|s|}\right)\\
&= \ex\left(\exp\left( s \cdot \sum_{|h|<p/2}\widetilde{\gamma_p}(h) \X(h)\right)\right)+  O\left(e^{-N}\right),
\end{align*}
as desired.
\end{proof}
Next, we compute the Laplace transform of the random variable $ \sum_{|h|<p/2}\widetilde{\gamma_p}(h) \X(h)$ at real numbers $s$ with $2\leq s\leq (\log p)^2$. 
\begin{pro}\label{AsympLaplace}
Let $p$ be a large prime and $2\leq s\leq (\log p)^2$ be a real number. Then we have
$$
\ex\left(\exp\left(s \cdot \sum_{|h|<p/2}\widetilde{\gamma_p}(h) \X(h)\right)\right)=\exp\left(\frac{2}{\pi}s\log s+ B_0 s +O(\log s)\right),
$$
where 
$$
B_0= \frac{2}{\pi}\left( \gamma+\log 2-\log \pi+ \frac{1}{2} \int_{0}^{\infty} \frac{f(u)}{u^2}du\right).
$$
\end{pro}
To prove this result we need the following elementary lemma, which follows from Lemma 3.3 of \cite{LRW}.
\begin{lem}[Lemma 3.3 of \cite{LRW}]\label{LogEx}
Let $f:[0, \infty)\to \mathbb{R}$ be defined by $$f(t):=\begin{cases} \log \ex (e^{t\X}) & \text{ if } 0\leq t <1,\\ \log \ex (e^{t\X})- 2t & \text{ if } t\geq 1,\end{cases}$$
where $\X$ is a random variable with Sato-Tate distribution on $[-2, 2]$. Then we have the following estimates
$$f(t)\ll \begin{cases} t^2 & \text{ if } 0\leq t<1,\\ \log(2 t)  & \text{ if } t\geq 1,\end{cases}$$
and 
$$ f'(t)\ll \begin{cases} t & \text{ if } 0< t<1,\\ t^{-1} & \text{ if } t> 1.\end{cases}$$
\end{lem}
\begin{proof}[Proof of Proposition \ref{AsympLaplace}]
By the independence of the $\X(h)$ we have
$$  \log \ex\left(\exp\left(s\sum_{|h|<p/2}\widetilde{\gamma_p}(h) \X(h)\right)\right)= \sum_{|h|<p/2}\log \ex\left(\exp\left(s\widetilde{\gamma_p}(h) \X(h)\right)\right).$$
Using the estimate \eqref{OddEvenGM} and Lemma \ref{LogEx} we obtain
$$\sum_{\substack{|h|<p/2\\ h \text{ even }}}\log \ex\left(\exp\left(s\widetilde{\gamma_p}(h) \X(h)\right)\right) \ll \sum_{\substack{|h|<p/2\\ h \text{ even }}} \frac{s^2}{p^2} \ll \frac{(\log p)^4}{p}. $$
We now restrict ourselves to the case $h= 2k+1$ is odd. First, it follows from \eqref{BGM} and Lemma \ref{LogEx}
that
$$  \sum_{|k|> s^2}  \log \ex\left(\exp\left(s\widetilde{\gamma_p}(2k+1) \X(2k+1)\right)\right) \ll \sum_{|k|> s^2} \frac{s^2}{k^2} \ll 1.$$
Moreover, when $ |k|<s^2$ we use \eqref{OddEvenGM} and Lemma \ref{LogEx} to get
$$
 \log \ex\left(\exp\left(s\widetilde{\gamma_p}(2k+1) \X(2k+1)\right)\right)= \log \ex\left(\exp\left(\frac{s}{(2k+1)\pi}\X(2k+1)\right)\right) +O\left(\frac{s}{p}\right).
$$
 Combining these estimates, and using Lemma \ref{LogEx} we obtain 
 \begin{equation}\label{ParSum}
\log \ex\left(\exp\left(s\sum_{|h|<p/2}\widetilde{\gamma_p}(h) \X(h)\right)\right)
= \frac{4s}{\pi} \sum_{2k+1\leq s/\pi}\frac{1}{2k+1} + 2\sum_{0\leq k<s^2} f\left(\frac{s}{(2 k+1)\pi}\right) +O(1)
\end{equation}
since $\X(h)$ and $-\X(h)$ have the same distribution. Next, we observe that
$$ \sum_{2k+1\leq s/\pi}\frac{1}{2k+1}= \frac{1}{2}\sum_{1\leq k\leq s/2\pi}\frac{1}{k} +\log 2+O\left(\frac{1}{s}\right)= \frac{\log s}{2} + \frac{1}{2} \left( \gamma+ \log 2- \log \pi\right)+ O\left(\frac{1}{s}\right).
$$
Furthermore, by partial summation and Lemma \ref{LogEx} we get
\begin{align*}
\sum_{0\leq k<s^2} f\left(\frac{s}{(2 k+1)\pi}\right) =\int_{0}^{s^2} f\left(\frac{s}{(2 u+1)\pi}\right)du + O(\log s). 
\end{align*}
Finally, making the change of variables $v=s/((2 u+1)\pi)$, the integral on the right hand side of this estimate becomes
$$ \frac{s}{2\pi} \int_{s/((2s^2+1)\pi)}^{s/\pi} \frac{f(v)}{v^2}dv= \frac{s}{2\pi}\int_{0}^{\infty} \frac{f(v)}{v^2}dv +O(\log s),
$$
by Lemma \ref{LogEx}. Inserting these estimates in \eqref{ParSum} completes the proof.
\end{proof}
Using the saddle-point method and Propositions \ref{Laplace} and \ref{AsympLaplace}, we prove Theorem \ref{LowerB}.
\begin{proof}[Proof Theorem \ref{LowerB}]
For a real number $t$, we define
\begin{equation}\label{LargeSumt}
\mathcal{N}_p(t):=\frac{1}{p-1} \left| \left \{a\in \Fs: \frac{1}{\sqrt{p}}  \im \sum_{0\leq n\leq p/2} e_p(n^3+an)> t\right \} \right|.
\end{equation}
 Let $ \mathcal{E}_p$ be the set in the statement of Proposition \ref{Laplace}, and $\widetilde{\mathcal{N}_p}(t)$ be the proportion of $a\in \Fs \setminus \mathcal{E}_p$ such that $\sum_{|h|<p/2} \widetilde{\gamma_p}(h)\B(a-h)>t$. Then, it follows from \eqref{EqualityB} that
$$ \mathcal{N}_p(t)= \widetilde{\mathcal{N}_p}(t) +O\left(p^{-1/10}\right).$$
Furthermore, it follows from Propositions \ref{Laplace} and \ref{AsympLaplace} that for all positive real numbers $s$ such that $2\leq s\leq (\log p)/(50\log \log p)^2$ we have 
\begin{equation}\label{EstDistri}
\begin{aligned}
\int_{-\infty}^{\infty} e^{st } \widetilde{\mathcal{N}_p}(t) dt & 
= \frac{1}{p-1}\sum_{a\in \Fs\setminus \mathcal{E}_p}\int_{-\infty}^{\sum_{|h|<p/2} \widetilde{\gamma_p}(h) \B(a-h)} e^{st }  dt\\
& = \frac{1}{s(p-1)} \sum_{a\in \Fs\setminus \mathcal{E}_p} \exp\left(s\sum_{|h|<p/2}\widetilde{\gamma_p}(h) \B(a-h)\right)\\
&=\exp\left(\frac{2}{\pi}s\log s+ B_0 s +O(\log s)\right).
\end{aligned} 
\end{equation}
 Let  $V$ be a large real number such that $V\leq \frac{2}{\pi}\log\log p-2\log\log\log p$. We shall choose $s$ (the saddle point) such that 
\begin{equation}\label{Saddle}
\left(\frac{2}{\pi} s\log s+B_0 s-sV\right)'=0 \Longleftrightarrow s= \exp\left(\frac{\pi}{2}V-\frac{\pi}{2}B_0-1\right).
\end{equation}
Let $0<\delta<1$ be a small parameter to be chosen, and put $S=s e^{\delta}$. Then, it follows from \eqref{EstDistri} that 
\begin{align*}
\int_{V+2\delta/\pi}^{\infty} e^{st } \widetilde{\mathcal{N}_p}(t)dt 
&\leq \exp\left(s(1-e^{\delta})(V+ 2\delta/\pi)\right)\int_{V+2\delta/\pi}^{\infty} e^{St }\widetilde{\mathcal{N}_p}(t)dt\\
&\leq  \exp\left(s(1-e^{\delta})(V+2\delta/\pi)+ \frac{2}{\pi}s e^{\delta} \log s+ \frac{2}{\pi} s e^{\delta} \delta +B_0 s e^{\delta} +O(\log s)\right) \\
&=  \exp\left(\frac{2}{\pi}s\log s+B_0 s+ \frac{2}{\pi} s(1+\delta-e^{\delta}) +O(\log s)\right).
\end{align*}
Therefore, choosing $\delta= C_0\sqrt{(\log s)/s}$ for a suitably large constant $C_0$ and using \eqref{EstDistri} we obtain
$$
\int_{V+2\delta/\pi}^{\infty} e^{st } \widetilde{\mathcal{N}_p}(t)dt \leq  e^{-V}\int_{-\infty}^{\infty} e^{st } \widetilde{\mathcal{N}_p}(t)dt. 
$$
A similar argument shows that
$$
\int_{-\infty}^{V-2\delta/\pi} e^{st } \widetilde{\mathcal{N}_p}(t)dt \leq  e^{-V}\int_{-\infty}^{\infty} e^{st } \widetilde{\mathcal{N}_p}(t)dt. 
$$
Combining these bounds with \eqref{EstDistri} gives
\begin{equation}\label{EstDistri2}
\int_{V-2\delta/\pi}^{V+2\delta/\pi} e^{st } \widetilde{\mathcal{N}_p}(t)dt =  \exp\left(\frac{2}{\pi} s\log s+ B_0 s +O(\log s)\right).
\end{equation}
Furthermore, since $\widetilde{\mathcal{N}_p}(t)$ is non-increasing as a function of $t$ we can bound the above integral as follows
$$ 
e^{sV+O(s\delta)}\widetilde{\mathcal{N}_p}(V+2\delta/\pi) \leq \int_{V-2\delta/\pi}^{V+2\delta/\pi} e^{st } \widetilde{\mathcal{N}_p}(t)dt\leq  e^{sV+O(s\delta)} \widetilde{\mathcal{N}_p}(V-2\delta/\pi).
$$ 
Inserting these bounds in \eqref{EstDistri2} and using the definition of $s$ in terms of $V$, we obtain
$$\widetilde{\mathcal{N}_p}(V+2\delta/\pi) \leq \exp\left(- \frac{2}{\pi} \exp\left(\frac{\pi}{2}V-\frac{\pi}{2}B_0-1\right)\big(1+O( \delta)\big)\right) \leq \widetilde{\mathcal{N}_p}(V-2\delta/\pi),
$$
and thus 
$$ 
\widetilde{\mathcal{N}_p}(V)= \exp\left(- \frac{2}{\pi} \exp\left(\frac{\pi}{2}V-\frac{\pi}{2}B_0-1\right)\left(1 + O\left(\sqrt{V}e^{-\pi V/4}\right)\right)\right).
$$
as desired.

\end{proof}

\end{document}